\documentclass[11pt]{amsart}

\usepackage{mathrsfs}
\usepackage{amsmath, amsthm, amssymb, amscd, amsgen,amsxtra,amsfonts, amsbsy}
\allowdisplaybreaks
\usepackage{verbatim}

\usepackage{tikz}
\usepackage{tikz-cd}
\usepackage{color}
\definecolor{shadecolor}{gray}{0.875}
\definecolor{dblue}{rgb}{0,0,.6}
\usepackage[colorlinks=true, linkcolor=dblue, citecolor=dblue, filecolor = dblue, menucolor = dblue, urlcolor = dblue]{hyperref}
\usepackage{mathtools}
\usepackage{extarrows}
\usepackage{ifthen}

\usepackage{graphicx}
\usepackage{subcaption}

\newcommand{\mathds}[1]{{\mathbb #1}}

\numberwithin{equation}{section}

\begin{document}
%
%
%
\theoremstyle{definition}
\newtheorem{Definition}{Definition}[section]
\newtheorem*{Definitionx}{Definition}
\newtheorem{Convention}{Definition}[section]
\newtheorem{Construction}{Construction}[section]
\newtheorem{Example}[Definition]{Example}
\newtheorem{Examples}[Definition]{Examples}
\newtheorem{Remark}[Definition]{Remark}
\newtheorem*{Remarkx}{Remark}
\newtheorem{Remarks}[Definition]{Remarks}
\newtheorem{Caution}[Definition]{Caution}
\newtheorem{Conjecture}[Definition]{Conjecture}
\newtheorem*{Conjecturex}{Conjecture}
\newtheorem{Question}[Definition]{Question}
\newtheorem{Hypothesis}[Definition]{Hypothesis}
\newtheorem*{Questionx}{Question}
\newtheorem*{Acknowledgements}{Acknowledgements}
\newtheorem*{Notation}{Notation}
\newtheorem*{Organization}{Organization}
\newtheorem*{Disclaimer}{Disclaimer}
\theoremstyle{plain}
\newtheorem{Theorem}[Definition]{Theorem}
\newtheorem*{Theoremx}{Theorem}

\newtheorem*{thmA}{Theorem A}
\newtheorem*{thmB}{Theorem B}

\newtheorem{Claim}[Definition]{Claim}
\newtheorem{Proposition}[Definition]{Proposition}
\newtheorem*{Propositionx}{Proposition}
\newtheorem{Lemma}[Definition]{Lemma}
\newtheorem{Corollary}[Definition]{Corollary}
\newtheorem*{Corollaryx}{Corollary}
\newtheorem{Fact}[Definition]{Fact}
\newtheorem{Facts}[Definition]{Facts}
\newtheoremstyle{voiditstyle}{3pt}{3pt}{\itshape}{\parindent}%
{\bfseries}{.}{ }{\thmnote{#3}}%
\theoremstyle{voiditstyle}
\newtheorem*{VoidItalic}{}
\newtheoremstyle{voidromstyle}{3pt}{3pt}{\rm}{\parindent}%
{\bfseries}{.}{ }{\thmnote{#3}}%
\theoremstyle{voidromstyle}
\newtheorem*{VoidRoman}{}

\newenvironment{specialproof}[1][\proofname]{\noindent\textit{#1.} }{\qed\medskip}
\newcommand{\blowup}{\rule[-3mm]{0mm}{0mm}}
\newcommand{\cal}{\mathcal}
\newcommand{\Aff}{{\mathds{A}}}
\newcommand{\BB}{{\mathds{B}}}
\newcommand{\CC}{{\mathds{C}}}
\newcommand{\BC}{{\mathds{C}}}
\newcommand{\EE}{{\mathds{E}}}
\newcommand{\FF}{{\mathds{F}}}
\newcommand{\GG}{{\mathds{G}}}
\newcommand{\HH}{\mathrm{H}}
\newcommand{\NN}{{\mathds{N}}}
\newcommand{\ZZ}{{\mathds{Z}}}
\newcommand{\PP}{{\mathds{P}}}
\newcommand{\QQ}{{\mathds{Q}}}
\newcommand{\BQ}{{\mathds{Q}}}
\newcommand{\RR}{{\mathds{R}}}
\newcommand{\Liea}{{\mathfrak a}}
\newcommand{\Lieb}{{\mathfrak b}}
\newcommand{\Lieg}{{\mathfrak g}}
\newcommand{\Liem}{{\mathfrak m}}
\newcommand{\ideala}{{\mathfrak a}}
\newcommand{\idealb}{{\mathfrak b}}
\newcommand{\idealg}{{\mathfrak g}}
\newcommand{\idealm}{{\mathfrak m}}
\newcommand{\idealp}{{\mathfrak p}}
\newcommand{\idealq}{{\mathfrak q}}
\newcommand{\idealI}{{\cal I}}
\newcommand{\lin}{\sim}
\newcommand{\num}{\equiv}
\newcommand{\dual}{\ast}
\newcommand{\iso}{\cong}
\newcommand{\homeo}{\approx}
\newcommand{\mm}{{\mathfrak m}}
\newcommand{\pp}{{\mathfrak p}}
\newcommand{\qq}{{\mathfrak q}}
\newcommand{\rr}{{\mathfrak r}}
\newcommand{\pP}{{\mathfrak P}}
\newcommand{\qQ}{{\mathfrak Q}}
\newcommand{\rR}{{\mathfrak R}}
\newcommand{\OO}{{\cal O}}
\newcommand{\CO}{{\cal O}}
\newcommand{\numero}{{n$^{\rm o}\:$}}
\newcommand{\mf}[1]{\mathfrak{#1}}
\newcommand{\mc}[1]{\mathcal{#1}}
\newcommand{\into}{{\hookrightarrow}}
\newcommand{\onto}{{\twoheadrightarrow}}
\newcommand{\Spec}{\operatorname{Spec}}
\newcommand{\BigSpec}{\operatorname{\mathbf{Spec}}}
\newcommand{\Spf}{\operatorname{Spf}}
\newcommand{\Proj}{\operatorname{Proj}}
\newcommand{\Pic}{\operatorname{Pic}}
\newcommand{\Mov}{\operatorname{Mov}}
\newcommand{\Nef}{\operatorname{Nef}}
\newcommand{\MW}{\operatorname{MW}}
\newcommand{\Br}{\operatorname{Br}}
\newcommand{\NS}{\operatorname{NS}}
\newcommand{\Gr}{\operatorname{Gr}}
\newcommand{\Lie}{\operatorname{Lie}}
\newcommand{\sm}{\operatorname{sm}}
\newcommand{\Sym}{\operatorname{Sym}}
\newcommand{\Aut}{\operatorname{Aut}}
\newcommand{\Autp}{\operatorname{Aut^p}}
\newcommand{\ord}{\operatorname{ord}}
\newcommand{\coker}{{\rm coker}\,}
\newcommand{\divisor}{\operatorname{div}}
\newcommand{\Def}{\operatorname{Def}}
\newcommand{\rank}{\mathop{\mathrm{rank}}\nolimits}
\newcommand{\Ext}{\mathop{\mathrm{Ext}}\nolimits}
\newcommand{\EXT}{\mathop{\mathscr{E}{\kern -2pt {xt}}}\nolimits}
\newcommand{\Hom}{\mathop{\mathrm{Hom}}\nolimits}
\newcommand{\HOM}{\mathop{\mathscr{H}{\kern -3pt {om}}}\nolimits}
\newcommand{\chari}{\mathop{\mathrm{char}}\nolimits}
\newcommand{\ch}{\mathop{\mathrm{ch}}\nolimits}
\newcommand{\CH}{\mathop{\mathrm{CH}}\nolimits}
\newcommand{\supp}{\mathop{\mathrm{supp}}\nolimits}
\newcommand{\codim}{\mathop{\mathrm{codim}}\nolimits}
\newcommand{\IMAGE}{\mathop{\mathrm{Im}}\nolimits}
\newcommand{\Span}{\mathop{\mathrm{Span}}\nolimits}
\newcommand{\DIV}{\mathop{\mathrm{div}}\nolimits}
\renewcommand{\tilde}{\widetilde}
\newcommand{\pr}{\mathrm{pr}}
\newcommand{\calA}{\mathscr{A}}
\newcommand{\calH}{\mathscr{H}}
\newcommand{\calL}{\mathscr{L}}
\newcommand{\calM}{\mathscr{M}}
\newcommand{\bcalM}{\overline{\mathscr{M}}}
\newcommand{\calN}{\mathscr{N}}
\newcommand{\calX}{\mathscr{X}}
\newcommand{\calK}{\mathscr{K}}
\newcommand{\calD}{\mathscr{D}}
\newcommand{\calY}{\mathscr{Y}}
\newcommand{\calC}{\mathscr{C}}
\newcommand{\CM}{\mathcal{M}}
\newcommand{\CK}{\mathcal{K}}
\newcommand{\CV}{\mathcal{V}}
\newcommand{\piet}{{\pi_1^{\rm \acute{e}t}}}
\newcommand{\Het}[1]{{H_{\rm \acute{e}t}^{{#1}}}}
\newcommand{\Hfl}[1]{{H_{\rm fl}^{{#1}}}}
\newcommand{\Hcris}[1]{{H_{\rm cris}^{{#1}}}}
\newcommand{\HdR}[1]{{H_{\rm dR}^{{#1}}}}
\newcommand{\hdR}[1]{{h_{\rm dR}^{{#1}}}}
\newcommand{\loc}{{\rm loc}}
\newcommand{\et}{{\rm \acute{e}t}}
\newcommand{\defin}[1]{{\bf #1}}

\newcommand{\alb}{\operatorname{alb}}
\newcommand{\Alb}{\operatorname{Alb}}
\newcommand{\St}{\mathrm{St}}

\ifthenelse{\equal{1}{1}}{
	\newcommand{\blue}[1]{{\color{blue}#1}}
	\newcommand{\green}[1]{{\color{green}#1}}
	\newcommand{\red}[1]{{\color{red}#1}}
	\newcommand{\purple}[1]{{\color{purple}#1}}
}{
	\newcommand{\blue}[1]{}
	\newcommand{\green}[1]{}
	\newcommand{\red}[1]{}
	\newcommand{\purple}[1]{}
}

\title{When is the diagonal contractible?}

\author{Xi Chen}
\address{632 Central Academic Building, University of Alberta, Edmonton, Alberta T6G 2G1, Canada}
\email{xichen@math.ualberta.ca}

\author{Frank Gounelas}
\address{Mathematisches Institut, Universit\"at Bonn, Endenicher Allee 60, 53115 Bonn, Germany}
\email{gounelas@math.uni-bonn.de}

\date{May 6, 2026}


\maketitle

\begin{abstract}
    For a smooth projective complex variety $X$, we study the problem of when there exists a
    birational morphism $X\times X\to Y$ to a projective variety $Y$ contracting the diagonal
    $\Delta_X\subset X\times X$ to a subvariety of smaller dimension. We prove this happens if and
    only if various conditions related to the Albanese morphism of $X$ are satisfied. We also give
    necessary and sufficient conditions for the existence of a contraction which is an
    isomorphism outside the diagonal and initiate the problem of understanding contractions of
    diagonals in higher products.
\end{abstract}

\section{Introduction}

Let $X$ be a smooth projective variety over the complex numbers $\CC$. The main purpose of this paper is to study the following problem.

\begin{Question}\label{question:diagonal contraction}
	Let $\Delta_X\subset X\times X$ be the diagonal. When does there
	exist a birational projective morphism $\pi\colon X\times X\to Y$ so that $\pi(\Delta_X)$ is a point?
\end{Question}

If the above happens, we say that the diagonal $\Delta_X\subset X\times X$ is \textit{contractible to a point}.
In this paper, in addition to contractions of the diagonal to a point, we consider contractions $\pi$ such that $\dim \pi(\Delta_X) < \dim X$.
\begin{Definition}
	We say that the diagonal $\Delta_X \subset X\times X$ is \textit{contractible to a variety of dimension $m$} or \textit{has an $m$-dimensional contraction} if there exists a birational projective morphism $\pi\colon X\times X\to Y$ such that $\dim \pi(\Delta_X) = m$.
\end{Definition}

\begin{Remark} 
	By taking the Stein factorisation, the existence of a birational $\pi$ contracting $\Delta_X$ is equivalent to the existence of a generically finite $\pi$ contracting $\Delta_X$, so we rather treat the stronger assumption as our main question.
\end{Remark}

In view of the classical case of $\dim X=1$ (see Section~\ref{sec:examples}), one would expect that
contractibility of the diagonal depends on subtle positivity properties of $X$. Instances of this
can already be found in \cite{lehmannottem}, where the positivity of the cycle
$[\Delta_X]\in\CH_{\dim X}(X\times X)$ is studied. We will see that the question of contractibility
of the diagonal is rather simple and in fact equivalent to properties of the Albanese morphism (cf.\
\cite[\S5.3]{lehmannottem}).

We now state the main theorems of this paper. For a smooth projective variety $X$, we denote the Albanese morphism by
\[
\begin{tikzcd}
	\alb\colon X \ar[r] & \Alb(X).
\end{tikzcd}
\]

For an abelian variety $A$, we may consider the \textit{difference map} 
\[
\begin{tikzcd}[row sep=tiny]
d\colon A\times A \ar[r] & A \\
(a_1, a_2) \ar[r, mapsto] & a_1 - a_2
\end{tikzcd}
\]
and the induced map 
\[
	\begin{tikzcd}[row sep=tiny]
		\lambda\colon X\times X \ar[rr, "\alb\times\alb"] & & \Alb(X)\times\Alb(X) \ar[r, "d"] & \Alb(X) \\
		(x_1, x_2) \ar[rrr, mapsto] &  &  &  \alb(x_1) - \alb(x_2).
	\end{tikzcd}
\]

\begin{Theorem}\label{thm:diagonal contraction}
	For a smooth projective variety $X$, the following are equivalent:
\begin{enumerate}
\item[(A)]
The diagonal $\Delta_X\subset X\times X$ is contractible to a point.
\item[(B)]
\begin{equation}\label{eq:contraction2point}
		\dim \lambda(X\times X) = 2\dim X.
	\end{equation}
\item[(C)]
\begin{equation}\label{eq:contraction2point2}
\dim \Alb(X) \ge 2\dim X = 2\dim \alb(X).
\end{equation}
and for every abelian subvariety $B\subset \Alb(X)$ and a general point $a\in \alb(X)$,
\begin{equation}\label{eq:contraction2point3} 
\dim B \ge 2 \dim ((B+a)\cap \alb(X))
\end{equation}
\item[(D)]
Two general points of $X$ impose independent conditions on the holomorphic cotangent bundle $\Omega_X^1$. That is, for two general points $x,y\in X$, the evaluation map
\[
\begin{tikzcd}
H^0(\Omega^1_X)\ar[two heads]{r} & \Omega^1_{X,x}\oplus\Omega^1_{X,y}
\end{tikzcd}
\]
is surjective. Or equivalently,
\begin{equation}\label{eq:contraction2point4} 
\operatorname{Image}\left(\wedge^n H^0(\Omega^1_X\otimes I_x\right) \to H^0(K_X)) \ne 0
\end{equation}
for a general point $x\in X$, where $n=\dim X$ and $I_x$ is the ideal sheaf of $x\in X$.
\end{enumerate}
\end{Theorem}

\begin{Remark}\label{rem:thm:diagonal contraction}
In the original version of this paper, we thought that \eqref{eq:contraction2point2} is sufficient
to guarantee \eqref{eq:contraction2point} (cf.\ Section~\ref{subsection:higherdim}). We are very grateful to the referee
for pointing out that this is not the case. When $\Alb(X)$ is not simple, we need
\eqref{eq:contraction2point3} in addition to \eqref{eq:contraction2point2}. The equivalence of
(B) and (C) is essentially a statement about abelian varieties, which is probably well known yet we
include a proof, whereas (B) and (D) being equivalent is automatic from the definitions.
\end{Remark}

To explain the conditions in the theorem, we say that a variety has \textit{maximal Albanese
dimension} if \[\dim X = \dim\alb(X),\] i.e., the Albanese morphism is generically finite onto its
image in the Albanese variety. This holds if and only if $\Omega^1_X$ is generically generated by
global sections, for example if it is globally generated - examples of such varieties are
subvarieties of abelian varieties. The first condition in \eqref{eq:contraction2point2} is about the \textit{irregularity} of $X$
\[
\begin{split}
q&=h^1(X,\OO_X)=h^0(X,\Omega^1_X)\\
&=\dim\Alb(X)\geq2\dim X.
\end{split}
\]
Condition (D) asks that this $q$-dimensional space of global $1$-forms evaluates surjectively at any
\emph{two} general points simultaneously, i.e., $\Omega^1_X$ is globally generated at general pairs
of points. Equivalently, the Gauss map $\gamma\colon\alb(X)\dashrightarrow\mathrm{Gr}(\dim X,q)$,
sending a smooth point to its tangent plane in $\mathrm{Lie}(\Alb(X))$, has the property that the
tangent planes at two general points have trivial intersection (see the proof of the implication (B)
$\implies$ (C) below).

Examples satisfying the conditions of the theorem are geometrically non-degenerate subvarieties in
abelian varieties of dimension at least twice their own dimension; see Sections \ref{sec:examples}
and \ref{subsection:higherdim} for precise definitions and examples, including general complete intersections of codimension $\ge \frac{1}{2}\dim\Alb(X)$. A famous example in the literature is the Fano surface of lines $S$ of a smooth cubic threefold $X$. Here $S\to \Alb(S)\cong J(X)$ is a closed embedding into the intermediate Jacobian of $X$, which is an abelian fivefold. We have $\dim J(X)=5\geq2\dim S$ and actually condition (B) holds: it is known by Clemens--Griffiths that the difference map
\[S\times S\to J(X)\]
has degree 6 onto its image, the (singular) theta divisor. 

As it is rather simple, we give a quick summary of the method of proof of the above theorem in the most important case, namely when $q=0$. If $f\colon X\times X\to Y$ is birational, then for any ample line bundle $L$ on $Y$, we have that $f^*L$ is big and nef. Given our assumption on $q$, we may write $f^*L=\pi_1^*D_1+\pi_2^*D_2$ for $D_i\in\Pic X$ and $\pi_i$ the two projections (in general there is a contribution from the Picard group of the Albanese variety, see Lemma~\ref{lemma:picard group decomposition}). Restricting to fibres of $\pi_i$, we see that $D_i$ must also be big and nef. If $f(\Delta_X)$ were a point, then $D_1+D_2=f^*L|_{\Delta_X}=0$ which gives a contradiction.

It is worth noting that the image $Y$ of the birational contraction is basically never smooth along the image of $\Delta_X$ (see \cite[Corollary 2.63]{Kollar-Mori}) - for example in the curve case (see Section~\ref{subsection:curves}) the image is singular but with Gorenstein singularities \cite[Corollary 2.5]{badescu}. 

We also prove a necessary and sufficient condition for $\Delta_X$ to be contractible to a variety of dimension $m<\dim X$, although it is less explicit than \eqref{eq:contraction2point}. The two statements could be combined but we preferred to give a cleaner statement of Theorem~\ref{thm:diagonal contraction}.

\begin{Theorem}\label{thm:contract2variety}
For a smooth projective variety $X$, the diagonal $\Delta_X\subset X\times X$ is contractible to a
variety of dimension $m$ if and only if there exists a dominant morphism $\pi\colon X\to W$ to a
projective variety $W$ with $\dim W = m$ such that $\dim \lambda (X_p\times X_p) = 2\dim X_p$ for a
general point $p\in W$ and $X_p = \pi^{-1}(p)$.
\end{Theorem}

For interesting examples with this behaviour one can consider products of curves, see Lemma~\ref{lem:productofcurves}.

We end with some questions and stronger variants. As we saw, when \eqref{eq:contraction2point} holds, $\lambda$ is generically finite over its image and it contracts $\Delta_X$ to a point. 

It is reasonable to ask whether this map is essentially ``the only'' morphism which contracts $\Delta_X$. To be more precise, let
\begin{equation}\label{eq:lambda}
	\begin{tikzcd}
		\lambda\colon X\times X\ar[r, "\lambda_\St"] & Y \ar{r} & \Alb(X)
	\end{tikzcd}
\end{equation}
be the Stein factorisation of $\lambda$ such that $X\times X\to Y$ is birational. 

\begin{Question} \label{question:throughdifference} 
	Is it true that every generically finite map $X\times X\to Y'$ contracting $\Delta_X$ to a point factors through $\lambda_\St$?
\end{Question}

We discuss this in Section~\ref{subsubsection:factoringthroughdifference} and in particular give a positive answer in some limited cases (see Lemma~\ref{lem:amplecotangent}).

There are other variants of Question~\ref{question:diagonal contraction}. One could additionally require that $\pi\colon X\times X \setminus \Delta_X\to Y\setminus \pi(\Delta_X)$ is an isomorphism. 
Note that if $X$ is a smooth projective variety which contains a rationally connected subvariety $Z$, e.g., a rational curve, then $\lambda$ contracts not only the diagonal $\Delta_X$, but also $Z\times Z$ to a point, since an abelian variety does not contain any rational curves. We can nevertheless prove the following criterion for this variant of the problem.

\begin{Theorem}\label{thm:isooutsidediagonal}
	For a smooth projective variety $X$, there exists a contraction $f\colon X\times X\to Y$ of the diagonal to a point such that \[\begin{tikzcd}
	X\times X\setminus\Delta_X \ar{r}{f} & Y\setminus f(\Delta_X)
	\end{tikzcd}\]
	is an isomorphism if and only if $\alb\colon X\to A = \Alb(X)$ is finite and birational over its image such that
	\begin{equation}\label{eq:dim1}
	\dim (X\times_A X \setminus\Delta_X) \le 0
	\end{equation}
	and for all $0\neq a\in A$ we have
	\begin{equation}\label{eq:dim2}
	\dim \left(\alb(X)\cap (\alb(X)+a)\right)\leq0.
	\end{equation}
\end{Theorem}

We explain the conditions a bit. Noting that $\lambda^{-1}(0)=X\times_A X$, condition \eqref{eq:dim1} is equivalent to the fact that $\alb$ is finite and only has finitely many fibres which contain more than one point, so is in particular birational too (i.e., $\alb$ is an isomorphism outside finitely many points of $X$). The second condition \eqref{eq:dim2} is equivalent to the fact that the difference map $d|_{\alb(X)\times\alb(X)}$ is finite outside the diagonal, i.e., in combination with \eqref{eq:dim1}, that $\lambda^{-1}(a)$ is finite for all $0\neq a\in A$.

These two conditions are true for subvarieties $X\subset A$ of abelian varieties,
satisfying $\dim X\cap(a+X)=0$ for all $a\neq0$, and we note that this last condition forces
condition \eqref{eq:contraction2point}. It would be interesting to give some examples of varieties $X$ with finite Albanese morphism $\alb\colon X\to A$ so that only finitely many fibres contain more than one point yet $\alb$ is not a closed embedding.

Another natural question is under what conditions we can contract the diagonals of $X^n = X\times X\times \ldots \times X$. We prove that at least we cannot contract all the diagonals simultaneously.

\begin{Theorem}\label{thm:multidiagonal}
For every smooth projective variety $X$ of $\dim X > 0$ and every $n\ge 3$, there does not exist a birational projective morphism
\[
\begin{tikzcd}
X^n = \underbrace{X\times X\times \cdots\times X}_n \ar{r}{f} & Y
\end{tikzcd}
\]
such that $f$ maps $\Delta_{ij}$ to a point for all $1\le i<j\le n$, where 
\[\Delta_{ij} = \{p\in X^n: \pi_i(p) = \pi_j(p)\}\]
with $\pi_i:X^n\to X$ being the $i$-th projection.
\end{Theorem}

\begin{Remark}
The above was instigated by the following question of Leon Takhtajan: given a smooth projective
curve $C$ and $n\geq2$, let $\Delta$ be the \textit{total diagonal} in the $n$-fold product
$C^n=C\times \cdots \times C$, i.e., the union of all multidiagonals. Does $C^n\setminus\Delta$ admit any
non-constant holomorphic functions? 
This would follow if all multidiagonals contracted to a subvariety of codimension two or larger.
\end{Remark}

\textbf{Notation:} For the remainder of the paper we will simplify notation and denote by $A=\Alb(X)$ and $\alpha=\alb$. We work over the complex numbers $\CC$ throughout, and varieties are reduced and irreducible.

\textbf{Acknowledgements:}
The second author heard the problem of when the diagonal of a variety might be contractible from
Robert Lazarsfeld. We would like to thank Gebhard Martin for a helpful conversation regarding
Proposition~\ref{prop:kappa=1} and the anonymous referee for spotting a mistake in a previous
version of Theorem~\ref{thm:diagonal contraction}. 

The first author is partially supported by Discovery Grant
RGPIN-2019-04775 from the Natural Sciences and Engineering Research Council of Canada. Both authors
would like to thank the ERC Synergy Grant HyperK (ID 854361) for its support, in particular for the
visit of the first author to the University of Bonn.

\section{Low dimensional examples}\label{sec:examples}

We first introduce two notions of non-degeneracy for subvarieties of abelian varieties that will feature throughout this section.

\begin{Definition}\label{def:nondeg}
An irreducible subvariety $Y\subset A$ of an abelian variety $A$ is \textit{non-degenerate} if it is not contained in any translate of a proper abelian subvariety of $A$, or equivalently, if the smallest abelian subvariety of $A$ containing $Y-y$ for some (equivalently, for a general) $y\in Y$ is $A$ itself.
\end{Definition}

\begin{Definition}{\cite[D\'efinition VIII.2.4]{debarretores}}\label{def:geodeg}
An irreducible subvariety $Y\subset A$ of an abelian variety $A$ is \textit{geometrically non-degenerate} if for every abelian subvariety $K\subset A$,
\[
\dim(Y+K)=\min(\dim Y+\dim K,\,\dim A).
\]
\end{Definition}
Equivalently, for every abelian quotient $\pi:A\to A'$ we have $\dim\pi(Y) = \min(\dim Y, \dim A')$. By \cite[Lemma II.12]{ransubvarieties}, this is also equivalent to the requirement that the kernel of the pullback map 
\[i^*\colon H^0(A,\Omega^d_A)\to H^0(\widetilde Y,\Omega^d_{\widetilde Y})\] contains no nonzero decomposable $d$-form, where $i\colon \widetilde Y\to Y\subset A$ is a resolution of singularities and $d=\dim Y$.

Geometric non-degeneracy implies non-degeneracy: if $Y\subset a+B$ for a proper abelian subvariety $B\subsetneq A$, then $Y+B\subset a+B$, so $\dim(Y+B)\le\dim B<\dim Y+\dim B$, a contradiction. The converse fails, as Examples \ref{ex:BxC} and \ref{ex:C1C2} below show. More importantly, the right condition for contractibility of the diagonal is (C): it is implied by geometric non-degeneracy (when $\dim A\ge 2\dim Y$), but does not follow from non-degeneracy alone. By the universal property of the Albanese, $\alpha(X)\subset A$ is always non-degenerate. For more on (geometrically) non-degenerate subvarieties of (semi)abelian varieties, see \cite{ransubvarieties,abramovichsubvarieties} as well as \cite[Chapter VIII]{debarretores}.

A third, strictly weaker notion of non-degeneracy is studied by Ran \cite[Definition II.1]{ransubvarieties}: $Y$ is \textit{nondegenerate in the sense of Ran} if its Gauss image lies on no Plücker hyperplane, or equivalently if the pullback $i^*\colon H^0(A,\Omega^d_A)\to H^0(\widetilde Y,\Omega^d_{\widetilde Y})$ is injective. This is implied by geometric non-degeneracy and we will not explore it further here.

The following sufficient condition will be the workhorse for verifying contractibility in the examples below: it reduces the problem to checking geometric non-degeneracy of the Albanese image. The key input is the following theorem of Debarre, which we will also use later in Section~\ref{sec:proofs} for the proof of (C) $\Rightarrow$ (B) in Theorem~\ref{thm:diagonal contraction}.

\begin{Theorem}[{\cite[Théorème VIII.2.2]{debarretores}}]\label{thm:debarre22}
Let $A$ be a complex torus, $X_1,X_2\subset A$ irreducible subvarieties, and
let $B$ be the largest subtorus of $A$ such that $X_1+X_2+B=X_1+X_2$. Write
$\pi:A\to A/B$ for the canonical surjection. Then
\[
\dim\bigl(\pi(X_1)+\pi(X_2)\bigr)=\dim \pi(X_1)+\dim \pi(X_2).
\]
\end{Theorem}

The following lemma is a consequence of the above which we will use repeatedly.

\begin{Lemma}\label{lem:dimYminusY}
	Let $Y\subset A$ be an irreducible subvariety of an abelian variety. Let $B:=\mathrm{Stab}^0(Y-Y)\subset A$ be the largest abelian subvariety with $(Y-Y)+B=Y-Y$, and let $\pi\colon A\to A/B$ be the quotient. Then
	\[
	\dim(Y-Y)=2\dim\pi(Y)+\dim B.
	\]
\end{Lemma}
\begin{proof}
	The subvariety $-Y\subset A$ is irreducible of dimension $\dim Y$, and $\pi(-Y)=-\pi(Y)$. Applying Theorem~\ref{thm:debarre22} to $X_1:=Y$, $X_2:=-Y$ gives $\dim(\pi(Y)-\pi(Y))=2\dim\pi(Y)$. Since $Y-Y$ is $B$-invariant by definition of $B$, we have $Y-Y=\pi^{-1}(\pi(Y-Y))$ and $\pi(Y-Y)=\pi(Y)-\pi(Y)$, so
	\[
	\dim(Y-Y)=\dim(\pi(Y)-\pi(Y))+\dim B=2\dim\pi(Y)+\dim B.\qedhere
	\]
\end{proof}

\begin{Proposition}\label{prop:geomnondeg-contracts}
	Let $X$ be a smooth projective variety with $\dim X=n$, and let $\alpha\colon X\to A=\Alb(X)$ be its Albanese morphism with image $Y$. Assume that 
	\begin{itemize}
		\item $\alpha$ is generically finite onto $Y$, 
		\item $Y$ is geometrically non-degenerate in $A$, 
		\item $q(X)\ge 2n$. 
	\end{itemize}  
	Then the conditions of Theorem~\ref{thm:diagonal contraction} hold; in particular, the diagonal of $X$ is contractible to a point.
\end{Proposition}
\begin{proof}
	It suffices to verify condition (B), namely $\dim(Y-Y)=2n$. The upper bound $\dim(Y-Y)\le 2n$ is immediate since $Y-Y$ is the image of $Y\times Y$ under the difference map. For the lower bound, let $B$ and $\pi\colon A\to A/B$ be as in Lemma~\ref{lem:dimYminusY}, so that $\dim(Y-Y)=2\dim\pi(Y)+\dim B$. Geometric non-degeneracy of $Y$ at $B$ gives $\dim\pi(Y)=\min(n,\,q-\dim B)$. If $\dim B\le q-n$ then $\dim(Y-Y)=2n+\dim B\ge 2n$; otherwise $\dim(Y-Y)=2q-\dim B\ge q\ge 2n$. In either case $\dim(Y-Y)=2n$.
\end{proof}

Turning now to the topic at hand, for which varieties is the diagonal contractible, the primary examples satisfying condition (C) are geometrically non-degenerate subvarieties $X\subset A$ of abelian varieties with $\dim A\ge 2\dim X$; see Example~\ref{ex:CI}. 

We note that the conditions in Theorem~\ref{thm:diagonal contraction} are invariant under birational morphisms from other smooth projective varieties, so we may as well try to classify minimal varieties with this property. 

We move now to some more precise statements in low dimensions, starting with curves.

\subsection{Curves}\label{subsection:curves} Let $C$ be a smooth projective curve. Then for $C=\Delta_C\subset C\times C$ the diagonal with normal bundle $N_{C/C\times C}$ we have 
\[\deg N_{C/C\times C} = \deg T_C = 2-2g.\]
In particular, if $g\leq1$ then $\Delta_C$ cannot be contracted by Grauert's contractibility criterion \cite[p.91, Theorem 2.1]{BHPV}. If $g\geq2$, then the difference map 
\begin{align*}
	\delta\colon C\times C &\to J(C) \\
	(p,q) &\mapsto \OO(p-q)
\end{align*}
always contracts $\Delta_C$ to a point, and is an isomorphism onto its image outside $\Delta_C$ if $C$ is not hyperelliptic, whereas it will have degree 2 if $C$ is hyperelliptic (see, e.g., \cite[p.223, Exercise D]{acgh1}). Hence, taking the Stein factorisation in the latter case, we see that the diagonal of a curve is contractible if and only if $g\geq2$. 

As we will see below in Lemma~\ref{lem:amplecotangent}, the difference map to the Jacobian is essentially the only morphism contracting the diagonal, in that any other birational morphism to a surface $f\colon C\times C\to S$ where $f(\Delta_C)$ is a point factors through the Stein factorisation of the difference map.

\subsection{Surfaces} 
Let $S$ be a smooth projective surface. From Theorem~\ref{thm:diagonal contraction}, the diagonal of
$S\times S$ is contractible to a point if and only if
the difference map $S\times S\to \Alb(S)$ has 4-dimensional image, which implies
\[\dim \alpha(S) = 2 \quad\text{and}\quad q(S)\ge 4.\]

In particular this is not possible if the Kodaira dimension $\kappa(X)\leq0$. Examples where it does happen are products of curves of genus at least two or more generally non-degenerate 2-dimensional subvarieties of abelian varieties of dimension at least 4. Here is a small classification result for surfaces of Kodaira dimension at most $1$ satisfying the numerical necessary conditions $q>2\dim S$ and generically finite Albanese.
\begin{Proposition}\label{prop:kappa=1}
	Let $S$ be a minimal smooth projective surface with Kodaira dimension $\kappa(S)\leq1$. Then
        condition \eqref{eq:contraction2point2} holds (i.e., $q(S)\geq4$ and the Albanese morphism of $S$ is generically
        finite) if and only if $S\cong (C\times E)/G$ for $G\subset E$ a finite subgroup acting diagonally on the product, where $E, C$ are smooth projective curves of genus $1$ and $g\geq3$ respectively.
\end{Proposition}
\begin{proof}
	One direction follows easily by computing
	\[ h^1((C\times E)/G, \OO_{(C\times E)/G})=g+1. \]
	Since the Albanese of $(C\times E)/G$ is $C'\times E$ (where $C'$ is the quotient of $C$ by the $G$-action on $C$) and the Albanese morphism is generically finite, the numerical conditions hold.

	For the other, note that from the classification of surfaces and Theorem~\ref{thm:diagonal contraction}, we have already ruled out $\kappa(S)\leq0$. If $\kappa(S)=1$, then $S$ admits an elliptic fibration
	\[ 
		\begin{tikzcd}
				f\colon S\ar[r] & C
		\end{tikzcd}
	\]
	to a smooth projective curve. From \cite[Lemma 3.4]{katsuraueno}, we obtain that $\dim\Alb(S)=g+1$ and that no fibre of $f$ gets contracted under the Albanese morphism. In particular, $f$ cannot contain any rational curves in its fibres (which would necessarily be contracted to $\Alb(S)$) and has only multiple fibres. Consider 
	\begin{equation*}
		\begin{tikzcd}
			J \arrow[d, "h"] \\
			C \arrow[r, "\psi"] & \overline{\mathcal{M}}_{1,1}
		\end{tikzcd}
	\end{equation*}
	 the corresponding Jacobian fibration of $f$ and the induced moduli map $\psi$. From \cite[Theorems 4.3.1, 4.3.20]{cossecdolgachevliedtke}, $h$ must be a smooth fibration. In particular $\psi$ is a contraction and it follows that $h$ must be isotrivial. This implies that all smooth fibres of $f$ are isomorphic. It follows that there is an elliptic curve $E\subset\mathrm{Bir}(S)$. The result now follows from \cite[Proposition 3.26]{fong} or \cite[\S 5]{rudakovshafarevichinseparable}.
\end{proof}

The surfaces $(C\times E)/G$ appearing in Proposition~\ref{prop:kappa=1} do \emph{not} satisfy condition (C) of Theorem~\ref{thm:diagonal contraction}, and in particular their diagonal is not contractible to a point. Indeed, the Albanese image $Y\cong C'\times E\subset\Alb(S)$ contains the elliptic factor $E$, so the abelian subvariety $B=\{0\}\times E$ satisfies $\dim B=1$ and $\dim((B+a)\cap Y)=1$ for any $a$, violating the requirement $\dim B\ge 2\dim((B+a)\cap Y)$ in condition (C). This is the same phenomenon illustrated in Example~\ref{ex:BxC}.

\begin{Corollary}
	Let $S$ be a smooth projective surface whose diagonal is contractible to a point. Then $S$ is of general type.
\end{Corollary}

Now we may ask the question when or whether it is possible at all that the diagonal of $S\times S$ is contractible to a curve.
From Theorem~\ref{thm:contract2variety}, the necessary and sufficient conditions for this to happen are
\begin{enumerate}
	\item there exists a projective dominant morphism to a curve $S\to W$,
	\item $\dim \alpha(S_p)=1$ for $p\in W$ general,
	\item $\alpha(S_p)$ is not an elliptic curve for $p\in W$ general.
\end{enumerate}
For example, if $C_1,C_2$ are two general curves of genus at least $2$, then the diagonal of $C_1\times C_2$ is contractible to a curve.

We end with a remark in a slightly different direction. In \cite{lehmannottem}, positivity properties of the cycle $[\Delta_X]\in\CH_n(X\times X)$ are discussed for surfaces. In particular, they study when $[\Delta_X]$ is big or nef. They prove for example that if $X$ is a K3 surface, then $[\Delta_X]$ is neither big nor nef, and that if $X$ is very general, then $\Delta_X$ is the unique effective cycle in its numerical class, so is extremal in the pseudo-effective cone of cycles. Nevertheless, by our theorem above, it cannot be contracted to a point.

\subsection{Higher dimension}\label{subsection:higherdim}

Other than high-codimensional non-degenerate subvarieties of abelian varieties, the following is easy to see and is the prototypical example in all dimensions of interesting contraction behaviour for the diagonal.

\begin{Lemma}\label{lem:productofcurves}
	Let $C_1,\ldots,C_n$ be general curves of genus $g_1,\ldots,g_n$ so that $g_i\geq2$ for all $i$. Then for any $0\leq t\leq n-1$, the diagonal $\Delta_{C_1\times\ldots\times C_n}\subset C_1\times\ldots\times C_n\times C_1\times\ldots\times C_n$ is contractible to a variety of dimension $t$.
\end{Lemma}
\begin{proof}
We use Theorem~\ref{thm:diagonal contraction} for $t=0$ and Theorem~\ref{thm:contract2variety} for $0<t\le n-1$.

\textit{Case $t=0$.} Set $X=C_1\times\cdots\times C_n$ and $A=J(C_1)\times\cdots\times J(C_n)$. The map $\lambda\colon X\times X\to A$ factors as a product of difference maps $\delta_i\colon C_i\times C_i\to J(C_i)$, $(x,y)\mapsto\alpha_i(x)-\alpha_i(y)$:
\[
\lambda\bigl((x_1,\ldots,x_n),(y_1,\ldots,y_n)\bigr)=\bigl(\delta_1(x_1,y_1),\ldots,\delta_n(x_n,y_n)\bigr).
\]
Hence $\lambda(X\times X)=(C_1-C_1)\times\cdots\times(C_n-C_n)$. Since $g_i\ge2$, the difference variety $C_i-C_i\subset J(C_i)$ has dimension $2$. Therefore $\dim\lambda(X\times X)=2n=2\dim X$, condition (B) holds, and the diagonal is contractible to a point.

\textit{Case $0<t\le n-1$.} By Theorem~\ref{thm:contract2variety}, it suffices to find a dominant morphism $\pi\colon X\to W$ with $\dim W=t$ such that $\lambda|_{X_p\times X_p}$ is generically finite onto its image for a general fibre $X_p$. Take $\pi$ to be the projection onto the first $t$ factors $C_1\times\cdots\times C_t$. Then $X_p\cong C_{t+1}\times\cdots\times C_n$, and the $t=0$ case applied to $C_{t+1},\ldots,C_n$ (each with $g_i\ge2$) gives $\dim\lambda(X_p\times X_p)=2(n-t)=2\dim X_p$, as required.
\end{proof}

It seems natural to wonder what the relation of contractibility of the diagonal and positivity
properties of the vector bundle $\Omega^1_X$ is. In the following, we provide some results in both
directions.

Even though we do not use it, we leave the following well-known lemma here to highlight further the
relevance of positivity properties of $\Omega^1_X$.

\begin{Lemma}
	Let $X$ be a smooth projective variety. Then the Albanese morphism $\alpha\colon X\to A$ is unramified (in particular finite) if and only if $\Omega^1_X$ is globally generated.
\end{Lemma}

\begin{Remark}\label{rem:artin} 
In view of results in \cite{artinalgebraization2} one could hope
that contractibility of the diagonal to a point could be related to $N_{\Delta/X\times
X}^*\cong\Omega^1_X$ being ample. This is only true in the divisor case, namely when $\dim X=1$ as
Section~\ref{subsection:curves} shows. In higher dimensions, ampleness of $\Omega^1_X$ is neither
necessary nor sufficient for contractibility. For necessity: a product $X=C_1\times C_2$ of two
curves of genus $g_i\geq2$ has
$\Omega^1_X=\mathrm{pr}_1^*\Omega^1_{C_1}\oplus\mathrm{pr}_2^*\Omega^1_{C_2}$ globally generated but
not ample (each summand restricts trivially along fibres of the other projection), yet the diagonal
of $X$ is contractible to a point by Lemma~\ref{lem:productofcurves}. For sufficiency: a
sufficiently general complete intersection of high degree and high codimension in projective space
has $\Omega^1_X$ ample yet $q=0$, so the diagonal cannot be contracted to a point.
\end{Remark}

In the other direction, we record a Hodge-theoretic sufficient condition for contractibility of the diagonal, of Castelnuovo--de Franchis type.
\begin{Proposition}\label{prop:cdf-higherdim}
	Let $X$ be a smooth projective variety of dimension $n$ with $q(X)\ge 2n$. Assume that for every $n$ linearly independent global $1$-forms $\omega_1,\dots,\omega_n\in H^0(\Omega^1_X)$, the product $\omega_1\wedge\dots\wedge\omega_n\in H^0(K_X)$ is nonzero. Then the conditions of Theorem~\ref{thm:diagonal contraction} hold; in particular, the diagonal of $X$ is contractible to a point.
\end{Proposition}
\begin{proof}
Note that $h^0(\Omega_X^1\otimes I_x) \ge h^0(\Omega_X^1) - n \ge n$. Then \eqref{eq:contraction2point4} follows.
\end{proof}

We list now three key examples attempting to give some insight into how contractibility is related to
non-degeneracy and positivity.

\begin{Example}\label{ex:BxC}
Let $B$ be an abelian variety with $\dim B\ge 1$ and $C$ a curve of genus $g\ge 2$. Set $X=B\times C$, so $A=\Alb(X)=B\times J(C)$ and $Y:=\alpha(X)=B\times C\subset A$.

The subvariety $Y$ is non-degenerate (it generates $A$), but it is not geometrically non-degenerate: for $K=B\times\{0\}$, one has $Y+K=Y$, so $\dim(Y+K)=\dim B+1$ while $\min(\dim Y+\dim K,\dim A)=\min(2\dim B+1,\dim B+g)>\dim B+1$ for $\dim B\ge 1$.

Moreover, $Y$ fails condition (C): for any nontrivial abelian subvariety $B'\subset B$, the general fibre of $Y\to A/(B'\times\{0\})$ is $B'\times\{{\rm pt}\}$ of dimension $\dim B'$, so condition (C) would require $\dim B'\ge 2\dim B'$, a contradiction. Consequently $\dim(Y-Y)=\dim B+2<2\dim B+2=2\dim Y$, condition (B) fails, and the diagonal is not contractible. The obstruction is the positive-dimensional stabiliser $\mathrm{Stab}^0(Y)=B$.
\end{Example}

\begin{Example}\label{ex:C1C2}
Let $C_1,C_2$ be curves of genera $g_1,g_2\ge 2$. Set $X=C_1\times C_2$, $A=J(C_1)\times J(C_2)$, and $Y=C_1\times C_2\subset A$.

The subvariety $Y$ satisfies conditions (B) and (C) (as established in Lemma \ref{lem:productofcurves}), so in particular it is non-degenerate with $\mathrm{Stab}^0(Y)=0$.

However, it is \emph{not} geometrically non-degenerate: consider the projection $\pi\colon A\to J(C_1)$ with kernel $K=\{0\}\times J(C_2)$. The image $\pi(Y)$ is $C_1$, which has dimension 1. Since $1 \neq \dim J(C_1)$ (as $g_1 \ge 2$) and $1 \neq \dim Y$ (which is 2), $Y$ fails the definition of geometric non-degeneracy. This example shows condition (C) is strictly weaker than geometric non-degeneracy.
\end{Example}

\begin{Example}\label{ex:CI}
Complete intersections of ample divisors $Y=H_1\cap\cdots\cap H_r\subset A$ of codimension $r\ge\tfrac{1}{2}\dim A$ are geometrically non-degenerate: for any abelian subvariety $K\subset A$ of dimension $k$, ampleness ensures that a general translate of $K$ meets $Y$ with the expected dimension $\max(0,k-r)$, so the image of $Y$ under $A\to A/K$ has dimension $\min(\dim Y, \dim A/K)$ and $\dim(Y+K)=\min(\dim Y+k,\dim A)$. Since $\dim A\ge 2\dim Y$, condition (C) follows, and such $X\subset A$ (with $\alpha$ the inclusion) have contractible diagonal.
\end{Example}

\subsubsection{Factoring through the difference map}\label{subsubsection:factoringthroughdifference} We discuss here Question~\ref{question:throughdifference}. We give an answer in the following, somewhat trivial, case.
\begin{Lemma}\label{lem:amplecotangent}
	Let $X$ be a smooth projective variety satisfying the conditions of Theorem~\ref{thm:isooutsidediagonal}. Then any contraction of the diagonal factors through the difference map to $A$.
\end{Lemma}
\begin{proof}
	The assumptions imply that the Stein factorisation of the difference map $\lambda_{\St}\colon X\times X\to Y$ of \eqref{eq:lambda} contracts the diagonal to a point and is an isomorphism outside the diagonal. The claim now follows from \cite[Lemma 1.15(b)]{debarrehigherdim}.
\end{proof}
In other words, if $f\colon X\times X\to Y'$ is a birational morphism such that $f(\Delta_X)$ is a point, then for $Y$ as in \eqref{eq:lambda} there exists a birational morphism $h\colon Y\to Y'$ so that
\[ 
\begin{tikzcd}
	X\times X \ar[r, "\lambda_\St"] \ar[rr, bend left=30, "f"] & Y \ar[r, "h"] & Y'.
\end{tikzcd}
\]

\section{Proofs of the theorems}\label{sec:proofs}

We will need some preparatory lemmas along the way. They are rather standard but we include proofs for completeness.

\begin{Lemma}\label{lemma:iffsemiample}
	Let $X$ be a smooth projective variety. Then the diagonal $\Delta_X$ is contractible to a variety of dimension $m<\dim X$ if and only if there exists a big and semi-ample divisor $D$ on $X\times X$ such that
	\begin{equation}\label{eq:intersection}
		D^m \Delta_X \ne 0 \quad\text{and}\quad D^{m+1} \Delta_X = 0.
	\end{equation}
\end{Lemma}
\begin{proof}
	If there exists a birational $\pi\colon X\times X\to Y$ with $\dim \pi(\Delta_X)=m$, then for any ample divisor $L$ on $Y$, we have that $D = \pi^*L$ is big and semi-ample and that \eqref{eq:intersection} holds.

	Conversely, if there exists a big and semi-ample divisor $D$ on $X\times X$ satisfying \eqref{eq:intersection}, then $|nD|$ gives a birational morphism $X\times X\to Y$ contracting $\Delta_X$ to a variety of dimension $m$ for some $n$. 
\end{proof}

We note that for any smooth projective variety $X$, $\alpha(X)\subset A$ is non-degenerate by the universal property of the Albanese.

\begin{Lemma}\label{lemma:picard group decomposition}
	Let $X$ be a smooth projective variety. Then every $D\in \Pic_\QQ(X\times X)$ can be written as
	\begin{equation}\label{eq:divisor decomposition}
		D = \pi_1^* D_1 + \pi_2^* D_2 - (\alpha\times \alpha)^* B
	\end{equation}
	for some $D_i\in \Pic_\QQ(X)$ and $B\in \Pic_\QQ(A\times A)$ satisfying that
	\begin{equation}\label{eq:orthogonality}
		(\alpha\times \alpha)^* B . \pi_i^*L = 0
	\end{equation}
	for all $L\in \Pic_\QQ(X)$ and $i=1,2$,
	where $\pi_i\colon X\times X\to X$ are the two projections.
\end{Lemma}
\begin{proof}
	It follows from the K\"unneth decomposition that
	\begin{equation}\label{eq:hodge}
		\begin{aligned}
			\HH^{1,1}(X\times X, \QQ)  = &\ \pi_1^* \HH^{1,1}(X, \QQ) \oplus \pi_2^* \HH^{1,1}(X, \QQ) \oplus
			\\
			                        & \quad (\HH^1(X,\QQ)\otimes \HH^1(X,\QQ)\cap \HH^{1,1}(X\times X))
		\end{aligned}
	\end{equation}
	where $\HH^{1,1}(Z, \QQ) = \HH^2(Z, \QQ)\cap \HH^{1,1}(Z)$ is the Hodge group of a smooth projective variety $Z$.
	By \eqref{eq:hodge}, we obtain
	\begin{equation}\label{eq:picard}
		\begin{aligned}
			\Pic_\QQ(X\times X)  = &\ \pi_1^* \Pic_\QQ(X) \oplus \pi_2^* \Pic_\QQ(X) \oplus
			\\
			                    & \quad \rho_X^*(\HH^1(X,\QQ)\otimes \HH^1(X,\QQ)\cap \HH^{1,1}(X\times X))
		\end{aligned}
	\end{equation}
	where $\rho_X$ is the surjection $\Pic_\QQ(X\times X) \twoheadrightarrow \HH^{1,1}(X\times X,\QQ)$. 

	We know that the pullback
	\[
		\begin{tikzcd}
			\alpha^*\colon\HH^1(A,\QQ)\ar{r}{\sim} & \HH^1(X,\QQ)
		\end{tikzcd}
	\]
	of $\alpha$ is an isomorphism of Hodge structures.
	Therefore,
	\[
		\begin{aligned}
			  \HH^1(X,\QQ)\otimes \HH^1(X,\QQ)\ \cap\ & \HH^{1,1}(X\times X)                            \\
			  = &\ \HH^1(X,\QQ)\otimes \HH^1(X,\QQ)                                                       \\
			 &\quad \cap \big(\HH^{1,0}(X)\otimes \HH^{0,1}(X)\oplus \HH^{0,1}(X)\otimes \HH^{1,0}(X)\big) \\
			  = &\ (\alpha\times \alpha)^* \big(\HH^1(A,\QQ)\otimes \HH^1(A,\QQ) \\
			 &\quad \cap (\HH^{1,0}(A)\otimes \HH^{0,1}(A)\oplus \HH^{0,1}(A)\otimes \HH^{1,0}(A))\big) \\
			  = &\ (\alpha\times \alpha)^* \big(\HH^1(A,\QQ)\otimes \HH^1(A,\QQ)\cap \HH^{1,1}(A\times A)\big).
		\end{aligned}
	\]
	Combining this with the diagram
	\[
		\begin{tikzcd}
			0 \ar{r} & \frac{\HH^{0,1}(A\times A)}{\HH^1(A\times A, \ZZ)}
			\ar{r} \ar{d}{(\alpha\times \alpha)^*}[left]{\sim}& \Pic(A\times A) \ar{d}{(\alpha\times \alpha)^*} \ar{r}{\rho_A} & \HH^2(A\times A, \ZZ)\ar{d}{(\alpha\times \alpha)^*}\\
			0 \ar{r} & \frac{\HH^{0,1}(X\times X)}{\HH^1(X\times X, \ZZ)}\ar{r} & \Pic(X\times X) \ar{r}{\rho_X} & \HH^2(X\times X, \ZZ)
		\end{tikzcd}
	\]
	we see that
	\[
		\begin{aligned}
			 \quad\rho_X^*\big(\HH^1(X,\QQ)\ \otimes &\ \HH^1(X,\QQ)\cap \HH^{1,1}(X\times X)\big) \\
			 &\ = (\alpha\times \alpha)^* \rho_A^*\big(\HH^1(A,\QQ)\otimes \HH^1(A,\QQ)\cap \HH^{1,1}(A\times A)\big)
		\end{aligned}
	\]
	where $\rho_A$ is the surjection $\Pic_\QQ(A\times A) \twoheadrightarrow \HH^{1,1}(A\times A,\QQ)$.
	Thus, we can rewrite \eqref{eq:picard} as
	\[
		\begin{aligned}
			\Pic_\QQ(X\times X)  = &\ \pi_1^* \Pic_\QQ(X) \oplus \pi_2^* \Pic_\QQ(X) \oplus \\
			&\quad (\alpha\times \alpha)^* \rho_A^*\big(\HH^1(A,\QQ)\otimes \HH^1(A,\QQ)\cap \HH^{1,1}(A\times A)\big).
	\end{aligned}
	\]
	Then \eqref{eq:divisor decomposition} follows. Equality \eqref{eq:orthogonality} holds since \eqref{eq:hodge} is an orthogonal decomposition.
\end{proof}

\begin{proof}[Proof of (A) $\Leftrightarrow$ (B) in Theorem~\ref{thm:diagonal contraction}]
	One direction of the theorem is easy. If \eqref{eq:contraction2point} holds, then there exists a generically finite map $f\colon X\times X\to Y$ such that $\dim f(\Delta_X) = 0$. Namely the map
	\[
		\begin{tikzcd}
			\lambda\colon X\times X \arrow[r, "\alpha\times\alpha"] & A\times A \arrow[r, "d"] & A,
		\end{tikzcd}
	\]
	where $d\colon A\times A\to A$ is \textit{the difference map} $(x,y)\mapsto x-y$, is such a map because $\lambda(\Delta_X) = \{0\}$ and $\lambda$ is generically finite onto its image by \eqref{eq:contraction2point}. In other words, as a contraction of the diagonal to a point we may take the Stein factorisation of $\lambda$.

	The other direction will follow from Proposition~\ref{prop:irreg} proved below.
\end{proof}

\begin{Proposition}\label{prop:irreg}
	For a smooth projective variety $X$ and a birational morphism $f\colon X\times X\to Y$, we have
		\[ \dim(f(\Delta_X))\geq\dim X - \frac{1}{2} \dim \lambda(X\times X). \]
	
	In particular, if the diagonal is contractible to a point, then
	\eqref{eq:contraction2point} must hold.
\end{Proposition}
\begin{proof}
	Suppose that $\Delta_X\subset X\times X$ is contractible to a variety of dimension $m$.
	As pointed out in Lemma~\ref{lemma:iffsemiample}, there exists a big and semi-ample divisor $D$ on $X\times X$ satisfying \eqref{eq:intersection}.
	By Lemma~\ref{lemma:picard group decomposition},
	\[	D = \pi_1^* D_1 + \pi_2^* D_2 - (\alpha\times \alpha)^* B \]
	for some $D_1, D_2\in \Pic_\QQ(X)$ and $B \in \Pic_\QQ(A\times A)$ satisfying \eqref{eq:orthogonality}.

	Since $D$ is big and nef and $(\pi_i^*D_i)|_{\pi_i^{-1}(p)}=0=((\alpha\times \alpha)^* B)|_{\pi_i^{-1}(p)}$, we see that both $D_1$ and $D_2$ are big and nef on $X$. Furthermore, we can choose $D$ such that $D_1 = D_2$ by applying the group action $\sigma\colon X\times X \to X\times X$ sending $\sigma(x_1,x_2) = (x_2,x_1)$. More precisely, let
	\[	G = D + \sigma^* D = \pi_1^* (D_1 + D_2) + \pi_2^* (D_1+ D_2) - (\alpha\times \alpha)^* (B + \sigma^* B) \]
	where we use $\sigma$ for both maps $X\times X\to X\times X$
	and $A\times A \to A\times A$ given by $\sigma(x_1,x_2) = (x_2,x_1)$.

	As a sum of big and semi-ample divisors, $G$ is also big and semi-ample on $X\times X$, and since
	\[	D\Big|_{\Delta_X} = \sigma^*(D) \Big|_{\Delta_X} \]
	in $\Pic_\QQ(\Delta_X) = \Pic_\QQ(X)$, we still have
	\[	G^m \Delta_X \ne 0 \quad\text{and}\quad G^{m+1} \Delta_X = 0 \]
	or equivalently,
	\begin{equation}\label{eq:intersection2}
		\delta_X^*(G)^m \ne 0 \quad\text{and}\quad \delta_X^*(G)^{m+1} = 0
	\end{equation}
	for $\delta_X\colon X\to X\times X$ given by $\delta_X(x) = (x,x)$.

	Restricting $G$ to $\Delta_X$, we have
	\begin{equation}\label{eq:intersection3}
		\delta_X^*(G) = 2(D_1+D_2) - \alpha^* \delta_A^*(B + \sigma^* B) = 2(D_1 + D_2 - \alpha^* M)
	\end{equation}
	in $\Pic_\QQ(X)$ for $\delta_A\colon A\to A\times A$ given by $\delta_A(a) = (a,a)$ and
	\[	M = \frac{1}{2}\delta_A^* (B + \sigma^* B). \]
	It follows from \eqref{eq:intersection2} and \eqref{eq:intersection3} that
	\begin{equation}\label{eq:intersection4}
		(D_1 + D_2 - \alpha^* M)^{m+1} = 0.
	\end{equation}

	Let us consider
	\[	P = p_1^* M + p_2^* M - (B + \sigma^* B) \]
	in $\Pic_\QQ (A\times A)$, where $p_i\colon A\times A\to A$ are the two projections.

	Let $d\colon A\times A\to A$ be the difference map $d(a,b) = a-b$. Then
	\[	P \Delta_A = P d^*(0) = 0. \]
	Since $d^*(0)$ and $d^*(a)$ are algebraically equivalent, we conclude that
	\[	P d^*(a)  = 0 \]
	for all $a\in A$. Then
	\begin{equation}\label{eq:intersection5}
		(\alpha\times \alpha)^* P . \lambda^*(a) = (\alpha\times\alpha)^* (P d^*(a)) = 0
	\end{equation}
	for all $a\in A$, where $\lambda = d\circ(\alpha \times \alpha)\colon X\times X\to A$.

	We claim that
	\begin{equation}\label{eq:intersection6}
		G^{2m+1} \lambda^*(a) = 0
	\end{equation}
	for all $a\in A$. This follows from
	\[
		\begin{aligned}
			G^{2m+1} \lambda^*(a) & = \big((G-(\alpha\times\alpha)^* P) + (\alpha\times\alpha)^* P\big)^{2m+1} \lambda^*(a)               \\
			                      & = (G-(\alpha\times\alpha)^* P)^{2m+1} \lambda^*(a) \hspace{24pt} \text{(by \eqref{eq:intersection5})} \\
			                      & = \big(\pi_1^*(D_1 + D_2 - \alpha^* M) + \pi_2^*(D_1 + D_2 - \alpha^* M)\big)^{2m+1} \lambda^*(a)     \\
			                      & = 0 \hspace{24pt} \text{(by \eqref{eq:intersection4})}.
		\end{aligned}
	\]

	Since $G$ is big and nef on $X\times X$, \eqref{eq:intersection6} implies that $\lambda\colon X\times X\to A$ has relative dimension at most $2m$.
	Hence
	\[
			2m \ge \dim(X\times X) - \dim \lambda(X\times X)
	\]
	and the proposition follows.
\end{proof}

\begin{proof}[Proof of (B) $\Rightarrow$ (C) in Theorem~\ref{thm:diagonal contraction}]
From the assumption, it is easy to see that \eqref{eq:contraction2point2} holds. Let $Y = \alpha(X)$. Clearly, $\lambda$ is generically one-to-one onto its image if and only if the map
\[
\begin{tikzcd}
T_{Y,a} \oplus T_{Y,b} \ar{r} & T_{A,a-b} \cong H^0(T_A)
\end{tikzcd}
\]
on tangent spaces given by $(v_1,v_2)\to v_1 - v_2$ is injective for two general points $a,b\in Y$. That is,
\begin{equation}\label{eq:tangent}
\dim (T_{Y,a} \cap T_{Y,b}) = 0
\end{equation}
where both $T_{Y,a}$ and $T_{Y,b}$ are considered as subspaces of $H^0(T_A)$.

Suppose that there exists an abelian subvariety $B\subset A$ such that
\[
\dim B < 2\dim ((B+a)\cap Y) = 2\dim ((B+b)\cap Y)
\]
Then
\[
\dim T_{B,0} < 2\dim (T_{B,0}\cap T_{Y,a}) = 2\dim (T_{B,0}\cap T_{Y,b})
\]
and it follows that
\[
\dim (T_{B,0}\cap T_{Y,a} \cap T_{Y,b}) > 0
\]
which contradicts \eqref{eq:tangent}. 
\end{proof}

From the above proof, we see that $\lambda$ is generically finite if and only if $\dim X = \dim Y$ and \eqref{eq:tangent} holds.

\begin{proof}[Proof of (C) $\Rightarrow$ (B) in Theorem~\ref{thm:diagonal contraction}]
Let $Y=\alpha(X)\subset A=\Alb(X)$, and let $B$ and $\pi\colon A\to A/B$ be as in Lemma~\ref{lem:dimYminusY}, so that
\[
\dim(Y-Y)=2\dim\pi(Y)+\dim B.
\]
Write $f$ for the dimension of a general fibre of $\pi|_Y\colon Y\to\pi(Y)$, so that $\dim\pi(Y)=\dim Y-f$. The general fibre of $\pi|_Y$ over $\pi(a)$ is $Y\cap(B+a)$, so $f=\dim\bigl((B+a)\cap Y\bigr)$ for general $a\in Y$. Condition (C) applied to the abelian subvariety $B$ gives $\dim B\ge 2f$, hence
\[
\dim(Y-Y)=2(\dim Y-f)+\dim B\ge 2\dim Y - 2f + 2f = 2\dim Y,
\]
as required.
\end{proof}

\begin{proof}[Proof of (B) $\Leftrightarrow$ (D)]
Since $\alpha^*\colon H^0(\Omega^1_A)\xrightarrow{\sim} H^0(\Omega^1_X)$ is an isomorphism and $\Omega^1_A$ is trivial on $A$, evaluation at a single point factors as
\[
H^0(\Omega^1_X)\xrightarrow{\sim} H^0(\Omega^1_A)\xrightarrow{\mathrm{ev}_{\alpha(x)}} \Omega^1_{A,\alpha(x)}\xrightarrow{d\alpha_x^*}\Omega^1_{X,x},
\]
where $\mathrm{ev}_{\alpha(x)}$ is an isomorphism. Thus, surjectivity of the composition at a single
general $x$ is equivalent to the pullback $d\alpha_x^*$ being surjective, or equivalently $d\alpha_x
\colon T_{X,x} \to T_{A,\alpha(x)}$ being injective, i.e., $X$ having maximal Albanese dimension. 

Similarly, the two-point evaluation map $H^0(\Omega^1_X) \to \Omega^1_{X,x} \oplus \Omega^1_{X,y}$
is surjective if and only if its dual map $T_{X,x} \oplus T_{X,y} \to H^0(T_A) \cong
\mathrm{Lie}\,A$, given by $(u, v) \mapsto d\alpha_x(u) + d\alpha_y(v)$, is injective. In other
words the images of the differentials intersect trivially, so the
subspaces $T_{Y,\alpha(x)}$ and $T_{Y,\alpha(y)}$ of $\mathrm{Lie}\,A$ intersect only at the origin,
which is exactly condition \eqref{eq:tangent}. As shown in the proof of (B)$\Rightarrow$(C) above,
this holds if and only if $\lambda$ is generically finite onto its image, i.e., condition (B).

Finally, the two formulations of (D) are equivalent. Forms in $H^0(\Omega^1_X\otimes I_x)$ are exactly those vanishing at $x$, so for any $\omega_1,\dots,\omega_n\in H^0(\Omega^1_X\otimes I_x)$, the wedge $\omega_1\wedge\dots\wedge\omega_n\in H^0(K_X)$ is nonzero if and only if at some general $y$ we have that $\omega_1(y),\dots,\omega_n(y)$ are linearly independent in $\Omega^1_{X,y}$, i.e.\ they span $\Omega^1_{X,y}$. Combined with surjectivity at the single point $x$ alone (i.e., maximal Albanese dimension, which is implied by either formulation), this is exactly surjectivity of $H^0(\Omega^1_X)\to\Omega^1_{X,x}\oplus\Omega^1_{X,y}$.
\end{proof}

\begin{proof}[Proof of Theorem~\ref{thm:contract2variety}]
Suppose that there exists a generically finite projective morphism $f\colon X\times X\to Y$ such that $\dim f(\Delta_X) = m$. Then it is clear that we have a dominant morphism $X\cong \Delta_X \to W = f(\Delta_X)$. Since $f$ contracts $\Delta_X$ to $W$, it contracts $\Delta_{X_p}\subset X_p\times X_p$ to a point for all $p\in W$. By the same argument as for Proposition~\ref{prop:irreg}, we see that there exists a big and nef divisor $G$ on $X\times X$ such that
\begin{equation}\label{eq:G}
G \lambda_p^*(a) = 0
\end{equation}
for $p\in W$ general and all $a\in A$, where $\lambda_p\colon X_p\times X_p\to A$ is the restriction of
$\lambda\colon X\times X\to A$ to $X_p\times X_p$.

We claim that $G$ is big and nef on $X_p\times X_p$ for $p\in W$ general. Otherwise,
\[
G^{2n-2m}.(X_p\times X_p) = 0
\]
for $n=\dim X$. Since $X_p\times X_p$ and $X_p\times X_q$ are algebraically equivalent for $p,q\in W$ general, we have
\[
G^{2n-2m}.(X_p\times X_q) = 0
\]
for $p,q\in W$ general.
This is impossible since $G$ is big on $X\times X$ and $X_p\times X_q$ covers $X\times X$. So $G$ is big and nef on $X_p\times X_p$ for $p\in W$ general. Combining this with \eqref{eq:G}, we see that $\lambda_p$ maps $X_p\times X_p$ generically finitely onto its image.

On the other hand, suppose that 
there is a dominant projective morphism $\pi\colon X\to W$ such that
$\dim W = m$ and $\lambda$ maps $X_p\times X_p$ generically finite onto its image.

Let us consider $\eta\colon X\times X \to W\times W\times A$ given by
\[\eta(x_1,x_2) = (\pi(x_1), \pi(x_2), \lambda(x_1,x_2)).\]
Clearly, $\eta$ contracts $\Delta_X$ to $\Delta_W \times \{0\}\cong W$. On the other hand, since $\eta$ maps the fibre of $X\times X$ over $(p,p)\in W\times W$ generically finitely onto its image, it is generically finite over its image.
\end{proof}

\begin{proof}[Proof of Theorem~\ref{thm:isooutsidediagonal}]
For the ``if'' part, we consider the map $\lambda\colon X\times X\to A$ defined in the proof of Proposition~\ref{prop:irreg}. 

From \eqref{eq:dim2}, the map $d\colon \alpha(X)\times \alpha(X) \to A$ is finite outside of the diagonal $\Delta_{\alpha(X)}$.
And since $\alpha\colon X\to A$ is finite over $\alpha(X)$,
the map
\[
\begin{tikzcd}
X\times X \setminus X\times_A X = X\times X \setminus \lambda^{-1}(0) \ar{r}{\lambda} & A\setminus \{0\}
\end{tikzcd}
\]
is finite. On the other hand, from \eqref{eq:dim1},
\[
X\times_A X = \Delta_X \sqcup \Sigma
\]
for a finite set $\Sigma$ of points.
It suffices to take $X\times X\xrightarrow{f} Y\to A$ to be the Stein factorisation of $\lambda$.

For the ``only if'' part, let $f\colon X\times X\to Y$ be a birational morphism which is an isomorphism outside the diagonal $\Delta_X$. We let $D = f^*L$ for an ample divisor $L$ on $Y$.
As in the proof of Proposition~\ref{prop:irreg},
we let
\[
\begin{aligned}
D &= \pi_1^* D_1 + \pi_2^* D_2 - (\alpha\times \alpha)^* B\\
G &= D + \sigma^* D\\
M &= \frac{1}{2}\delta_A^* (B + \sigma^* B)\\
P &= p_1^* M + p_2^* M - (B + \sigma^* B).
\end{aligned}
\]
From \eqref{eq:intersection4},
\[
D_1 + D_2 - \alpha^* M = 0
\]
and hence
\[
G - (\alpha\times \alpha)^* P = \pi_1^*(D_1 + D_2 - \alpha^* M) + \pi_2^*(D_1 + D_2 - \alpha^* M) = 0.
\]
In addition, since $f$ does not contract any curve outside of $\Delta_X$, we have
\[
G . C > 0
\]
for all irreducible curves $C\not\subset \Delta_X$. 

If \eqref{eq:dim1} fails, then there exists an irreducible curve $C\subset X\times_A X = \lambda^{-1}(0)$ such that $C\not\subset \Delta_X$, which implies
\[
(\alpha\times \alpha)^*P . C = G.C > 0.
\]
On the other hand, since $(\alpha\times \alpha)(C) \subset \Delta_A$ and $P\Delta_A = 0$,
we have
\[
P . (\alpha\times \alpha)_* C = 0,
\]
which is a contradiction, giving \eqref{eq:dim1}. It is easy to see that this implies that $\alpha\colon X\to A$ is finite and birational over its image.

Similarly, if \eqref{eq:dim2} fails for some $a\ne 0$, then there exists an irreducible curve $C\subset \lambda^{-1}(a)$ and so
\[
(\alpha\times \alpha)^*P . C = G.C > 0.
\]
On the other hand, since $(\alpha\times \alpha)(C) \subset d^{-1}(a)$ and $Pd^* (a) = 0$,
we have
\[
P . (\alpha\times \alpha)_* C = 0,
\]
which is a contradiction, giving \eqref{eq:dim2}.
\end{proof}

\begin{proof}[Proof of Theorem~\ref{thm:multidiagonal}]
Using the same argument for Lemma~\ref{lemma:picard group decomposition} and the Theorem of the Cube, we see that every $D\in \Pic_\QQ(X^n)$ can be written as
\[
D = \sum_{i=1}^n \pi_i^* D_i - \sum_{1\le i<j\le n} \pi_{ij}^* (\alpha\times\alpha)^* B_{ij}
\]
for some $D_i\in \Pic_\QQ(X)$ and $B_{ij}\in \rho_A^*\big(\HH^1(A,\QQ)\otimes \HH^1(A,\QQ)\cap \HH^{1,1}(A\times A)\big)$,
where $\alpha\colon X\to A$ is the albanese map, $\pi_{ij}\colon X^n\to X\times X$ is the $ij$-th projection and $\rho_A$ is the surjection $\Pic_\QQ(A\times A) \twoheadrightarrow \HH^{1,1}(A\times A,\QQ)$.

Suppose that such $f\colon X^n \to Y$ exists. Let $D = f^* L$ for some ample divisor $L$ on $Y$. Then
$D$ is big and semi-ample and
\[
D\Big|_{\Delta_{ij}} = 0
\]
for all $1\le i<j\le n$. Let
\[
G = \sum_{\sigma\in \Sigma_n} \sigma^* D
\]
where $\Sigma_n$ is the symmetric group on $n$ letters and $\sigma\colon X^n\to X^n$ is the action of $\sigma\in \Sigma_n$ on $X^n$.

Clearly, $G$ remains big and semi-ample,
\[
G\Big|_{\Delta_{ij}} = 0
\]
for all $1\le i<j\le n$ and we have a decomposition
\[
G = \sum_{i=1}^n \pi_i^* M - \sum_{1\le i<j\le n} \pi_{ij}^* (\alpha\times \alpha)^* B
\]
for some $M\in \Pic_\QQ(X)$ and $B\in \rho_A^*\big(\HH^1(A,\QQ)\otimes \HH^1(A,\QQ)\cap \HH^{1,1}(A\times A)\big)$ satisfying that $B = \sigma^* B$ for $\sigma\colon A\times A \to A\times A$ given by $\sigma(a,b) = (b,a)$.

By restricting $G$ to $X\times (p_2,\ldots,p_n)$ for $(p_2,\ldots,p_n)\in X^{n-1}$ general, we see that $M$ is big and nef and by restricting it to
\[
\Delta_{12} \cap \bigcap_{i=3}^n \pi_i^{-1}(p_i)
\]
for $(p_3,\ldots,p_n)\in X^{n-2}$, we see that
\[
2M = \delta_X^* (\alpha\times\alpha)^* B
\]
where $\delta_X\colon X\to X\times X$ is the map $\delta_X(p) = (p,p)$.
On the other hand, by restricting $G$ to
\[
\Delta_{123} \cap \bigcap_{i=4}^n \pi_i^{-1}(p_i)
\]
for $(p_4,\ldots,p_n)\in X^{n-3}$, we see that
\[
3M = 3\delta_X^* (\alpha\times \alpha)^* B
\]
which is impossible since $M$ is big. 
\end{proof}

\bibliographystyle{amsalpha}

\providecommand{\bysame}{\leavevmode\hbox to3em{\hrulefill}\thinspace}
\providecommand{\MR}{\relax\ifhmode\unskip\space\fi MR }
\providecommand{\MRhref}[2]{%
  \href{http://www.ams.org/mathscinet-getitem?mr=#1}{#2}
}
\providecommand{\href}[2]{#2}

\end{document}